\documentclass{article}
\usepackage{amsmath,amsthm,amsfonts,amssymb}
\usepackage{fullpage,microtype,url}
\usepackage{graphicx}

\DeclareMathOperator{\Cyl}{Cyl}

\theoremstyle{definition}
\newtheorem{theorem}{Theorem}

\newtheorem{lemma}[theorem]{Lemma}

\newtheorem{conjecture}[theorem]{Conjecture}

\newcommand{\e}{\mathrm{e}}
\DeclareMathOperator*{\Exp}{\mathbb{E}}
\newcommand{\dzeta}{\mathrm{d}\zeta}
\newcommand{\eps}{\varepsilon}

\newcommand{\Gtilde}{\tilde{G}}
\newcommand{\SVC}{S}
\newcommand{\lfm}{f}
\newcommand{\abs}[1]{\left| #1 \right|}
\newcommand{\zetamax}{\zeta_{\max}}
\newcommand{\Bin}{\mathrm{Bin}}

\renewcommand{\emptyset}{\varnothing}

{\unskip\nobreak\hskip 2em plus 1fil\nobreak$\Box$
\parfillskip=0pt%
\endtrivlist}

\begin{document}

\title{Tight bounds on the threshold for permuted $k$-colorability}

\author{
Varsha Dani \\ Computer Science Department \\ University of New Mexico 
\and
Cristopher Moore \\ Computer Science Department \\ University of New Mexico \\ and Santa Fe Institute 
\and
Anna Olson \\ Computer Science Department \\ University of Chicago
}

\maketitle 

\abstract{If each edge $(u,v)$ of a graph $G=(V,E)$ is decorated with a 
permutation $\pi_{u,v}$ of $k$ objects, we say that it has a 
\emph{permuted $k$-coloring} if there is a coloring 
$\sigma:V \to \{1,\ldots,k\}$ such that $\sigma(v) \ne \pi_{u,v}(\sigma(u))$ 
for all $(u,v) \in E$.  Based on arguments from statistical physics, 
we conjecture that the threshold $d_k$ for permuted $k$-colorability 
in random graphs $G(n,m=dn/2)$, where the permutations on the edges 
are uniformly random, is equal to the threshold for standard graph 
$k$-colorability.  The additional symmetry provided by random permutations 
makes it easier to prove bounds on $d_k$.  By applying the second moment 
method with these additional symmetries, and applying the first moment 
method to a random variable that depends on the number of available colors 
at each vertex, we bound the threshold within an additive constant.  
Specifically, we show that for any constant $\eps > 0$, for sufficiently 
large $k$ we have
\[
2 k \ln k - \ln k - 2 - \eps \le d_k \le 2 k \ln k - \ln k - 1 + \eps \, . 
\]
In contrast, the best known bounds on $d_k$ for standard $k$-colorability 
leave an additive gap of about $\ln k$ between the upper and lower bounds.}

\section{Introduction}

We consider random graphs $G(n,m)$ with $n$ vertices and $m$ edges
chosen uniformly without replacement.  We give each edge $(u,v)$ an
arbitrary orientation, and then associate it with a uniformly random
permutation $\pi_{u,v} \in S_k$, where $S_k$ denotes the group of
permutations of $k$ objects.  A \emph{permuted $k$-coloring} of this
decorated graph is a function $\sigma: V \to \{1,\ldots,k\}$ such
that $\sigma(v) \ne \pi_{u,v}(\sigma(u))$ for all edges $(u,v)$.  For
convenience we will sometimes reverse the orientation of an edge, and
write $\pi_{v,u} = \pi_{u,v}^{-1}$ when $u$ and $v$ are distinct.

We conjecture that there is a sharp threshold for the existence of 
such a coloring in terms of the average degree $d=2m/n$:
\begin{conjecture}
\label{conj:threshold}
For each $k \ge 3$ there is a constant $d_k$ such that
\[
\lim_{n \to \infty} \Pr\left[\mbox{$G(n,m=dn/2)$ has a permuted $k$-coloring}\right] 
= \begin{cases} 1 & d < d_k \\ 0 & d > d_k \, , \end{cases}
\]
\end{conjecture}
\noindent
where the probability space includes both the graph $G(n,m)$ and the set of permutations $\{\pi_{u,v}\}$.  Moreover, we conjecture that $d_k$ is also the threshold for standard graph $k$-colorability, which is the special case where $\pi_{u,v}$ is the identity permutation for all $u,v$:
\begin{conjecture}
\label{conj:same}
For the same $d_k$ as in Conjecture~\ref{conj:threshold}, 
\[
\lim_{n \to \infty} \Pr\left[\mbox{$G(n,m=dn/2)$ has a standard $k$-coloring}
\right] 
= \begin{cases} 1 & d < d_k \\ 0 & d > d_k \, , \end{cases}
\]
\end{conjecture}

Why might these two thresholds be the same?  First note that if $G$ is a tree, we can ``unwind'' the permutations on the edges, changing them all to the identity, by permuting the colors at each vertex.  For any set of permutations $\{\pi_{u,v}\}$, this gives a one-to-one map from permuted colorings to standard colorings.  Since sparse random graphs are locally treelike, for almost all vertices $v$ we can do this transformation on a neighborhood of radius $\Theta(\log n)$ around $v$.  The effect on $v$'s neighborhood of the other vertices' colors is then the same as it would be in standard graph coloring, except that the colors on the boundary are randomly permuted.  

In particular, suppose we choose a uniformly random coloring of a tree, erase the colors in its interior, and choose a new uniformly random coloring with the same boundary conditions.  The \emph{reconstruction threshold} is the degree $d$ above which this new coloring retains a significant amount of information about the original coloring~\cite{bhatnagar,sly}, and it is closely related to the clustering transition~\cite{recon-cluster}.  Since we can unwind the permutations on a tree, permuted $k$-colorability and standard $k$-colorability trivially have the same reconstruction threshold.

For another argument, consider the following alternate way to choose the permutations on the edges.  First choose a uniformly random permutation $\pi_u$ at each vertex $u$.  Then, on each edge $(u,v)$, let $\pi_{u,v} = \pi_u^{-1} \pi_v$.  (Algebraically, $\pi_{u,v}:E \to S_k$ is the coboundary of $\pi_u:V \to S_k$.)  Since $c(u) = \pi_{u,v}(c(v))$ if and only if $\pi_u(c(v)) = \pi_v(c(v))$, a local change of variables again gives a one-to-one map from permuted colorings to standard ones.  Now note that choosing $\pi_{u,v}$ in this way yields a uniform joint distribution on any set of edges that does not include a cycle; for instance, on a graph of girth $g$ the $\pi_{u,v}$ are $(g-1)$-wise independent and uniform.  Since most loops in a sparse random graph are long, we might hope that this distribution on the $\{\pi_{u,v}\}$ is the same, for all practical purposes, as the uniform distribution.

Finally, perhaps the most convincing argument for Conjecture~\ref{conj:same} comes from statistical physics.  Using cavity field equations to analyze the asymptotic behavior of message-passing algorithms such as belief propagation and survey propagation, we can derive thresholds for satisfiability or colorability~\cite{mezard-zecchina,mertens-mezard-zecchina,mulet}, as well as other thresholds such as clustering, condensation, and freezing~\cite{lenka-col,krzakala:etal:pnas}.  However, assuming that there is an equal density of vertices of each color, the cavity field equations for $k$-coloring in a random graph are identical to those for permuted $k$-coloring~\cite{lenka-potts,lenka-boettcher}: they simply express the fact that each edge $(u,v)$ forbids $u$ from taking a single color that depends on the color of $v$.

Thus if the physics picture is correct---and parts of it have been shown rigorously (e.g.~\cite{ach-aco-rt,amin-catch,amin-lenka})---then colorings and permuted colorings have the same ``thermodynamics'' on sparse random graphs, and hence the same thresholds for colorability, as well as for clustering, condensation, and freezing.

Conjecture~\ref{conj:same} is attractive because it is easier, given current methods, to prove tight bounds on the threshold for permuted $k$-coloring than it is for standard $k$-coloring.  
First we recall a simple upper bound.  Let $X$ denote the number of permuted $k$-colorings.  Since the permutations are chosen independently, the probability that any given coloring $\sigma$ is proper is $(1-1/k)^m$.  (Indeed, this is true of any multigraph with $m$ edges.) Thus the expected number of colorings is
\begin{equation}
\label{eq:expx}
\Exp[X] = k^n (1-1/k)^m = \left[ k (1-1/k)^{d/2} \right]^n \, .
\end{equation}
This is exponentially small if $k (1-1/k)^{d/2} < 1$, in which case $X=0$ with high probability by Markov's inequality.  Thus 
\begin{equation}
\label{eq:upper}
d_k \le \frac{2 \ln k}{-\ln (1-1/k)} < 2 k \ln k - \ln k \, . 
\end{equation}

Using the second moment method, we will prove a lower bound on $d_k$ that is an additive constant below this upper bound.  We also improve the upper bound, using a random variable that depends on the number of available colors at each vertex.  Our results show that, for any constant $\eps > 0$ and $k$ sufficiently large, 
\[
2 k \ln k - \ln k - 2 - \eps \le d_k \le 2 k \ln k - \ln k - 1 + \eps \, . 
\]
In contrast, the best known lower bound on the threshold for standard $k$-colorability is roughly $\ln k$ below the first moment upper bound.

To simplify our arguments, we work in a modified random graph model $\Gtilde(n,m)$ where the $m$ edges are chosen uniformly with replacement, and the endpoints of each edge are chosen uniformly with replacement from the $n$ vertices.  As a consequence, both self-loops and multiple edges occur with nonzero probability.  Note that, unlike standard $k$-colorability, a self-loop at a vertex $v$ does not necessarily render the graph uncolorable: it simply means that $\sigma(v)$ cannot be a fixed point of the permutation on the loop, i.e., $\sigma(v) \ne \pi_{v,v}(\sigma(v))$.  

However, if $\pi_{v,v}$ is the identity then coloring is impossible, and this occurs with probability $1/k!$.  Similarly, if $u$ and $v$ have $k$ edges between them, then with constant probability the permutations on these edges make a coloring impossible.  As a consequence, the probability that $\Gtilde(n,m)$ with random permutations has a permuted $k$-coloring is bounded below $1$.  

Our bounds proceed by showing that $\Gtilde(n,m)$ is permuted-$k$-colorable with probability $\Omega(1)$ if $d$ is sufficiently small, and is not permuted-$k$-colorable with high probability if $d$ is sufficiently large.  In the sparse case $m=O(n)$, $\Gtilde(n,m)$ is simple with probability $\Omega(1)$, in which case it coincides with $G(n,m)$.  Thus, assuming that Conjecture~\ref{conj:threshold} is true, these values of $d$ are bounds on the threshold $d_k$ for $G(n,m)$.

The rest of the paper is organized as follows.  In Section~\ref{sec:second} we give our second moment lower bound.  In Section~\ref{sec:first} we give our upper bound, which uses a random variable that depends on the number of available colors.  In Section~\ref{sec:iso}, we prove an isoperimetric inequality relevant to this random variable.  We defer the parts of our proofs that are ``mere calculus'' to Section~\ref{sec:calculus}.

\section{The second moment lower bound}
\label{sec:second}

As in the simple first moment upper bound, let $X$ denote the number of permuted $k$-colorings of a random multigraph $\Gtilde(n,m)$ with uniformly random permutations on its edges.  The Cauchy-Schwarz inequality applied to $X \cdot \; \mathbf{1}_{\{X > 0\}}$  gives
\[
\Pr[X > 0] \ge \frac{\Exp[X]^2}{\Exp[X^2]} \, . 
\]
Our goal is to show that $\Exp[X^2] / \Exp[X]^2 = O(1)$, so that a permuted coloring exists with probability $\Omega(1)$, for a certain value of $d$.  Assuming Conjecture~\ref{conj:threshold}, \emph{i.e.,} that a threshold $d_k$ exists, any such $d$ is a lower bound on $d_k$.

Computing the second moment $\Exp[X^2]$ requires us to sum, over all pairs of colorings $\sigma, \tau$, the probability $P(\sigma,\tau)$ that both $\sigma$ and $\tau$ are proper.  Since the edges of $\Gtilde$ and their permutations are chosen independently, we have
\[
P(\sigma,\tau) = p(\sigma, \tau)^m \, ,
\]
where $p(\sigma, \tau)$ is the probability that a random edge $(u,v)$, with a random permutation $\pi$, is satisfied by both colorings.  That is,
\[
p(\sigma, \tau) = \Pr_{u,v,\pi}\left[ \sigma(u) \ne \pi(\sigma(v)) \text{ and } \tau(u) \ne \pi(\tau(v)) \right] \, . 
\]
For random constraint satisfaction problems where each variable takes one of two values, such as $k$-SAT or hypergraph $2$-coloring~\cite{ach-moore,ach-moore-hyp,ach-peres}, $p(\sigma, \tau)$ is a function $p(\zeta)$ just of the overlap between $\sigma$ and $\tau$, i.e., the fraction $\zeta$ of variables on which they agree.  The second moment can then be bounded by maximizing a function of $\zeta$, which is typically a simple calculus problem.  

For pairs of $k$-colorings, however, $p(\sigma, \tau)$ depends on a $k$-by-$k$ matrix of overlaps, where $\zeta_{i,j}$ is the fraction of vertices $v$ such that $\sigma(v)=i$ and $\tau(v)=j$.  Computing the second moment then requires us to bound a function of roughly $k^2$ variables, a difficult high-dimensional maximization problem.  Achlioptas and Naor~\cite{ach-naor} used convexity arguments to bound this function on the Birkhoff polytope, showing that 
\[
d_k \ge 2 (k-1) \ln (k-1) \approx 2 k \ln k - 2 \ln k \, . 
\]
This leaves an additive gap of about $\ln k$ between the upper and lower bounds.  Note, however, that this bound is tight enough to determine, almost surely, the chromatic number $\chi(G)$ as a function of the average degree to one of two possible integers, namely $k$ or $k+1$ where $k$ is the smallest integer such that $2k \ln k > d$.  Achlioptas and Moore extended these arguments to random regular graphs~\cite{ach-moore-reg}, determining $\chi(G)$ as a function of $d$ to $k$, $k+1$, or $k+2$.


For permuted colorings, the second moment calculation is much easier.  The random permutations create additional local symmetries, making $p(\sigma, \tau)$ a function only of the fraction $\zeta$ on which the two colorings agree.  Thus we just have to maximize a function of a single variable.  As a consequence, we can prove a lower bound on $d_k$ that matches the upper bound~\eqref{eq:upper} within an additive constant.

\begin{theorem}
For any $\eps > 0$, for sufficiently large $k$ we have
\begin{equation}
\label{eq:lower}
d_k > 2 k \ln k - \ln k - 2 - \eps \, . 
\end{equation}
\end{theorem}

\begin{proof}
To compute the second moment, we sum over all $k^n {n \choose k} (k-1)^{n-z}$ pairs of colorings that agree at $z$ of the $n$ vertices.  We say that such a pair has overlap $\zeta = z/n$.  Then
\begin{equation}
\label{eq:second-sum}
\Exp[X^2] = k^n \sum_{z=0}^n {n \choose z} (k-1)^{n-z} \,p(z/n)^m \, , 
\end{equation}
where $p(\zeta)$ is the probability that a random edge, with a random permutation, is satisfied by both colorings.  Inclusion-exclusion gives
\[
p(\zeta) = \zeta^2 \left( 1-\frac{1}{k} \right) + 2 \zeta (1-\zeta) \left( 1-\frac{2}{k} \right) + (1-\zeta)^2 \left( 1-\frac{2}{k}+\frac{1}{k(k-1)} \right) \, . 
\]
Note that 
\[
p(1/k) = \left( 1-\frac{1}{k} \right)^2 \, , 
\]
corresponding to the fact that two independently random colorings typically have overlap $\zeta = 1/k+o(1)$.

We proceed as in~\cite{ach-moore}.  Approximating the sum~\eqref{eq:second-sum} with an integral and using~\eqref{eq:expx} gives
\begin{equation}
\label{eq:second-int}
\frac{\Exp[X^2]}{\Exp[X]^2}
\sim \frac{1}{\sqrt{n}} \sum_{z=0}^n \e^{\phi(z/n) n}
\sim \sqrt{n} \int_0^1 \dzeta \,\e^{\phi(\zeta) n} \, , 
\end{equation}
where $\sim$ hides multiplicative constants, where 
\[
\phi(\zeta) = h(\zeta) + (1-\zeta) \ln (k-1) - \ln k + \frac{d}{2} \ln \frac{p(\zeta)}{(1-1/k)^2} \, , 
\]
and where $h(\zeta) = -\zeta \ln \zeta - (1-\zeta) \ln (1-\zeta)$ is the entropy function.  Applying Laplace's method to the integral~\eqref{eq:second-int} then gives
\[
\frac{\Exp[X^2]}{\Exp[X]^2} \sim \frac{\e^{\phi(\zetamax) n}}{\sqrt{|\phi''(\zetamax)|}} \, , 
\]
where $\zetamax = \textrm{argmax}_{\zeta \in [0,1]} \phi(\zeta)$ is the global maximum of $\phi(\zeta)$, assuming that it is unique and that $\phi''(\zetamax) < 0$.  

We have $\phi(1/k) = 0$, so if $\zetamax = 1/k$ and $\phi''(1/k) < 0$ then $\Exp[X^2] / \Exp[X]^2 = O(1)$ and $\Pr[X > 0] = \Omega(1)$. The proof is then completed by the following lemma:
\begin{lemma}
\label{lem:psi}
For any constant $\eps > 0$, if $d = 2 k \ln k - \ln k - 2 - \eps$ and $k$ is sufficiently large, $\phi''(1/k) < 0$ and $\phi(\zeta) < 0$ for all $\zeta \ne 1/k$.
\end{lemma}
\noindent
We defer the proof of this lemma to Section~\ref{sec:calculus}.
\end{proof}

\section{An improved first moment upper bound}
\label{sec:first}

In this section we apply the first moment method to a weighted random variable, and improve the upper bound~\eqref{eq:upper} on $d_k$ by a constant.  Specifically, we will prove the following theorem:
\begin{theorem}
\label{thm:upper}
For any $\eps > 0$, for sufficiently large $k$ we have
\begin{equation}
\label{eq:betterupper}
d_k < 2 k \ln k - \ln k - 1 + \eps \, . 
\end{equation}
\end{theorem}

We define our random variable as follows.  Every coloring (proper or not) of $n$ vertices with $k$ colors is an element of $[k]^n$ where $[k]=\{1,2,\dots,k\}$.  Thus the set of colorings is an $n$-cube of side $k$, with a dimension for each vertex.  The set of \emph{proper} colorings is some subset of this cube, $\SVC \subset [k]^n$.  The classic first moment argument we reviewed above computes the expected number of proper (permuted) colorings, $X=|S|$.  Here we define a new random variable, where each proper coloring is given a weight that depends on the ``degree of freedom'' at each vertex, i.e., the number of colors that vertex could take if the colors of all other vertices stayed fixed.

For any proper coloring $\sigma$, for each vertex $v$, let $c(\sigma, v)$ denote the number of colors available for $v$ if its neighbors are colored according to $\sigma$.  That is, 
\[
c(\sigma, v) = k - \abs{ \{ \pi_{u,v}(\sigma(u)) \,|\, (u,v) \in E \} } \, . 
\]
Note that if there are multiple edges between $u$ and $v$, they can each forbid $v$ from taking a color.  If $v$ has a self-loop, we think of it as denying a color to itself, in each direction:
\begin{equation}
\label{eq:self-loop}
c(\sigma, v) = k - \abs{ \{ \pi_{u,v}(\sigma(u)) \,|\, (u,v) \in E, u \ne v \} \cup \{ \pi_{v,v}(\sigma(v)), \pi_{v,v}^{-1}(\sigma(v)) } \, . 
\end{equation}
If $\sigma$ is a proper coloring, we have $c(\sigma, v) \ge 1$ for all $v$, since $v$'s current color $\sigma(v)$ is available.  Let 
\[
w(\sigma) 
= \begin{cases}\prod_{v} \left( 1 / c(\sigma, v) \right) 
& \mbox{if $\sigma$ is proper} \\
0 & \mbox{otherwise} \, , 
\end{cases}
\]
and let 
\[
Z= \sum_{\sigma \in [k]^n} w(\sigma) \, . 
\]

Why this random variable?  The expected number of colorings $\Exp[X]$ can be exponentially large, even above the threshold where $\Pr[X > 0]$ is exponentially small.  But close to the threshold, solutions come in clusters, where some vertices are free to flip back and forth between several available colors.  In a cartoon where each cluster is literally a subcube of $[k]^n$, a cluster containing a coloring $\sigma$ contributes $\prod_v c(\sigma,v)$ to $X$, but only $1$ to $Z$.  Thus, roughly speaking, $Z$ counts the number of clusters rather than the number of colorings.  Since the sizes of the clusters vary, $Z$ has smaller fluctuations than $X$ does, and hence gives a tighter upper bound on $d_k$.  We note that ``cluster counting'' random variables of other sorts have been studied elsewhere, such as satisfying assignments with a typical fraction of free variables~\cite{amin-catch} and certain kinds of partial assignments~\cite{maneva-sinclair}.

In this same cartoon where clusters are subcubes, $Z$ also counts the number of \emph{locally maximal} colorings, i.e., those colorings where no vertex $v$ can be flipped to a ``higher'' color $q > \sigma(v)$, since such colorings correspond to the highest corner of the cluster.  Bounds on $d_3$ were derived by counting locally maximal 3-colorings in~\cite{ach-molloy,kaporis,dubois-mandler}, culminating in $d_3 \le 4.937$.  The bounds we derive below are slightly weaker for $k=3$, yielding $d_3 \le 5.011$, since we treat the degrees of the vertices as independent rather than conditioning on the degree distribution.  Nevertheless, computing $\Exp[Z]$ gives a considerably simpler argument for general $k$.

Clearly $Z > 0$ if and only if $\SVC \ne \emptyset$.  However, in Section~\ref{sec:iso} we prove the following:
\begin{lemma} 
\label{lem:iso}
If $\SVC \ne \emptyset$ then $Z \ge 1$. 
\end{lemma}
\noindent
Applying Markov's inequality, we see that 
\[
\Pr\left[ \mbox{$G$ has a permuted $k$-coloring} \right] = 
\Pr\left[Z \ge 1\right] 
\le \Exp[Z] \, ,
\]
and any $d$ such that $\Exp[Z] < 1$ is an upper bound on the threshold $d_k$.  Thus we will prove Theorem~\ref{thm:upper} by computing $\Exp[Z]$.

Given the symmetry provided by the random edge permutations, the expected weight $\Exp[w(\sigma)]$ of any given coloring is independent of $\sigma$.  By linearity of expectation and the fact that any given $\sigma$ is proper with probability $(1-1/k)^m$, we then have
\[
\Exp[Z] 
= k^n \,\Exp[w(\sigma)] 
= k^n (1-1/k)^m \,\Exp[w(\sigma) \,|\, \mbox{$\sigma$ proper} ] \, .
\]
Thus we are interested in the conditional expectation
\[
\Exp[w(\sigma) \,|\, \mbox{$\sigma$ proper} ] 
= \Exp\left[ \left. \prod_{v} \frac{1}{c(\sigma, v)} \right| \mbox{$\sigma$ proper} \right]  \, .
\]

\begin{lemma}
\label{lem:availablecolors}
Let $\sigma$ be a permuted $k$-coloring. For any vertex $v$, the conditional distribution of the number of available colors $c(\sigma, v)$ is a function only of $v$'s degree. In particular it does not depend on $\sigma$.
\end{lemma}

\begin{proof} For each neighbor $u$ of $v$, there is a uniformly random permutation $\pi = \pi_{u,v}$ such that $v$ is blocked from having the color $\pi(\sigma(u))$.  Since $\sigma$ is proper, we know that $\pi(\sigma(u)) \ne \sigma(v)$.  Subject to this condition, $\pi$ is uniformly random among the permutations such that $\pi(\sigma(u)) \ne \sigma(v)$, and thus $\pi(\sigma(u))$ is uniformly random among the colors other than $\sigma(v)$.  In particular, it does not depend on $\sigma(u)$. 

We can think of the forbidden colors $\pi(\sigma(u))$ as balls, and the colors other than $\sigma(v)$ as bins.  We toss $\deg v$ balls independently and uniformly into these $k-1$ bins, one for each edge $(u,v)$.  Then $c(v)$ is the number of empty bins, plus one for $\sigma(v)$.  
\end{proof}

\begin{lemma}
\label{lem:indep}
Let $\sigma$ be a permuted $k$-coloring, $(u, v)$ an edge, and $\pi =\pi_{u,v}$ the associated random permutation. Then $\pi(\sigma(u))$ and $\pi^{-1}(\sigma(v))$ are independent, and are uniform over $[k]-\sigma(v)$ and $[k]-\sigma(u)$ respectively.
\end{lemma}

\begin{proof}
Since $\sigma$ is proper, the conditional distribution of $\pi$ is uniform among all permutations such that $\pi(\sigma(u)) \ne \sigma(v)$ and $\pi^{-1}(\sigma(v)) \ne \sigma(u)$.  For any pair of colors $q, q'$ with $q \ne \sigma(v)$ and $q' \ne \sigma(u)$, there are exactly $(k-2)!$ permutations $\pi$ such that $\pi(\sigma(u))=q$ and $\pi^{-1}(\sigma(v))=q'$.  Thus all such pairs $(q,q')$ are equally likely, and the pair $\big( \pi(\sigma(u)), \pi^{-1}(\sigma(v)) \big)$ is uniform in $([k]-\sigma(v)) \times ([k]-\sigma(u))$.
\end{proof}

Note that Lemmas~\ref{lem:availablecolors} and~\ref{lem:indep} apply even to self-loops.  That is, if $\pi = \pi_{v,v}$ is uniformly random, then $\pi(\sigma(v))$ and $\pi^{-1}(\sigma(v))$ are independent and uniform in $[k]-\sigma(v)$.  Thus a self-loop corresponds to two balls, each of which can forbid a color.  Since a self-loop increases $v$'s degree by $2$, it has the same effect as two edges incident to $v$ would have.  Indeed, this is why we defined $c(\sigma,v)$ as in~\eqref{eq:self-loop}.

Now let $\{\deg v \,|\, v \in V\}$ denote the degree sequence of $G$.  By Lemmas~\ref{lem:availablecolors} and~\ref{lem:indep}, the numbers of available colors at the vertices $c(\sigma, v)$ are conditionally independent if their degrees are fixed.  Thus 
\begin{align*}
\Exp[w(\sigma) \,|\, \mbox{$\sigma$ proper}, \{\deg v\} ] 
&= \Exp\left[ \left. \prod_{v} \frac{1}{c(\sigma, v)} \,\right| \mbox{$\sigma$ proper}, \{\deg v\} \right] \\
&= \prod_v \Exp\left[ \left. \frac{1}{c(\sigma,v)} \,\right| \mbox{$\sigma$ proper}, \deg v \right] \\
&= \prod_v \sum_{c=1}^k \frac{Q(\deg v, k, c)}{c} \, ,
\end{align*}
where $Q(\deg v,k,c)$ denotes the probability that $v$ has $c$ available colors if it has $\deg v$; that is, the probability that if we toss $\deg v$ balls into $k-1$ bins, then $c-1$ bins will be empty.  Thus
\[
\Exp[w(\sigma) \,|\, \mbox{$\sigma$ proper} ] 
= \Exp_{\{\deg v\}} \prod_v \sum_{c=1}^k \frac{Q(\deg v, k, c)}{c} \, ,
\]
where the expectation is taken over the distribution of degree sequences in $\Gtilde(n,m)$.


The degree of any particular vertex in $\Gtilde(n,m=dn/2)$ is asymptotically Poisson with mean $d$.  
The degrees of different vertices are almost independent, as the next lemma shows.
\begin{lemma}\label{lem:degdist}
The joint probability distribution of the degree sequence of $\Gtilde(n,m=dn/2)$ is the same as that of $n$ independent 
Poisson random variables of mean $d$, conditioned on their sum being $2m$.
\end{lemma}
 
\begin{proof}
We can generate $\Gtilde(n,m=dn/2)$ as follows.  There are $n$ bins, one for each vertex.  We throw $2m$ balls uniformly and independently into the bins, and pair up consecutive balls to define the edges of the graph. 
The degree of each vertex is the number of balls in the corresponding bin.  The joint distribution of these occupancies is the product of $n$ independent Poisson distributions with mean $d$, conditioned on the total number of balls being $2m$; see e.g.~\cite[Theorem 5.6]{mitz-upf}.  
\end{proof}

The sum of $n$ independent Poisson variables of mean $d$ equals its mean $nd=2m$ with probability $O(1/\sqrt{m})=O(1/\sqrt{n})$.  Conditioning on an event that holds with probability $P$ increases the expectation by at most $1/P$, so
\begin{align}
\Exp[w(\sigma) \,|\, \mbox{$\sigma$ proper} ] 
&= \Exp_{\{\deg v\}} \prod_v \sum_{c=1}^k \frac{Q(\deg v, k, c)}{j} 
\nonumber \\
&= O(\sqrt{n}) \left( \Exp_{\deg v} \sum_{c=1}^k \frac{Q(\deg v, k, c)}{c} \right)^{\!n} \, ,
\label{eq:sqrtn}
\end{align}
where in the second line $\deg v$ is Poisson with mean $d$.  Since our goal is to show that $\Exp[Z]$ is exponentially small, the $\sqrt{n}$ factor will be negligible.

Now consider a balls and bins process with $k-1$ bins.  If the total number of balls is Poisson with mean $d$, then the number of balls in each bin is Poisson with mean $d/(k-1)$, and these are independent.  The probability that any given bin is empty, i.e., that any given color $q \ne \sigma(v)$ is available, is 
\[
r = \e^{-d/(k-1)} \, .
\]
The number of empty bins is binomially distributed as $\Bin(k-1,r)$, so
\[
Q(\deg v,k,c) = {k-1 \choose c-1} \,r^{c-1} \,(1-r)^{k-c} \, , 
\]
and
\begin{align*}
\Exp_{\deg v} \sum_{c=1}^k \frac{Q(\deg v, k, c)}{c} 
&= \sum_{c=1}^k \frac{1}{c} {k-1 \choose c-1} \,r^{c-1} \,(1-r)^{k-c} \\
&= \frac{1}{kr} \sum_{j=1}^k {k \choose c} \,r^c \,(1-r)^{k-c} \\
&= \frac{1}{kr} \left( 1-(1-r)^k \right) \, .
\end{align*}

Putting everything together, we have
\begin{align*}
\Exp[Z] & = k^n \left(1-\frac{1}{k}\right)^m \Exp[w(\sigma) \,|\, \mbox{$\sigma$ proper} ] \\
&= O(\sqrt{n}) \,k^n \left( 1-\frac{1}{k} \right)^{dn/2} 
\left( \frac{1}{kr} \left(1-(1-r)^k \right)\right)^n \\
& = O(\sqrt{n}) \left(1-\frac{1}{k} \right)^{dn/2} \e^{dn/(k-1)}
\left(1-(1-\e^{-d/(k-1)})^k \right)^n \, . 
\end{align*}
Taking the logarithm and dividing by $n$, in which case we can ignore the polynomial term $\sqrt{n}$, yields the following function of $d$:
\begin{equation}
\label{eqn:logEZ}
\lfm(d) := \lim_{n \to \infty} \frac{\ln \Exp[Z]}{n} 
= \frac{d}{2} \ln\left(1-\frac{1}{k}\right) + \frac{d}{k-1} + 
\ln\left( 1-\left(1-\e^{-d/(k-1)}\right)^k\right) \, .
\end{equation}
If $\lfm(d) < 0$ then $\Exp[Z]$ is exponentially small, so any such $d$ is an upper bound on $d_k$.  We prove the following lemma in Section~\ref{sec:calculus}:
\begin{lemma}
\label{lem:lfm}
For any constant $\eps > 0$, if $d = 2 k \ln k - \ln k - 1 + \eps$ and $k$ is sufficiently large, then $\lfm(d) < 0$.
\end{lemma}
\noindent
This completes the proof of Theorem~\ref{thm:upper}.





\section{An isoperimetric inequality}
\label{sec:iso}

In this section we prove Lemma~\ref{lem:iso}.  It has nothing to do with colorings; it is simply a kind of isoperimetric inequality that applies to any subset of $[k]^n$.  If $k=2$, it is the classic isoperimetric inequality on the Boolean $n$-cube.  That is, given $S \subseteq \{0,1\}^n$, for each $\sigma \in S$ let $\partial(\sigma)$ be the set of neighbors $\sigma' \in S$ that differ from $\sigma$ on a single bit.  Then
\[
\sum_{\sigma \in S} 2^{-|\partial(\sigma)|} \ge 1
\]
if and only if $S \ne \emptyset$.

First, some notation.  Let $V=\{1,\ldots,n\}$, and think of each element of $[k]^n$ as a function $\sigma:V \to [k]$.  Let $S \subseteq [k]^n$.  For each $\sigma \in S$ and $1 \le v \le n$, let $\partial_S(\sigma,v)$ denote the set of elements of $S$ that are ``neighbors of $\sigma$ along the $v$ axis,'' i.e., that agree with $\sigma$ everywhere other than at $v$.  That is,
\[
\partial_S(\sigma,v) = \left\{ \sigma' \in S \mid \forall u \ne v : \sigma'(u) = \sigma(u) \right\} \, .
\]
Let $c_S(\sigma,v)$ denote the number of such neighbors,
\[
c_S(\sigma,v) = \abs{\partial_S(\sigma,v)} \, ,
\]
and define the weight function $w_S$ as follows:
\[
w_S(\sigma) 
= \begin{cases}\prod_{v} \left( 1 / c_S(\sigma, v) \right) 
& \mbox{if $\sigma \in S$} \\
0 & \mbox{if $\sigma \notin S$} \, .
\end{cases}
\]
Then define the weight of the entire set as
\[
Z(S) = \sum_{\sigma \in [k]^n} w_S(\sigma).  
\]

\begin{lemma}
\label{lem:iso0}
If $S \ne \emptyset$ then $Z(S) \ge 1$.
\end{lemma}

%

\begin{proof}
If $S = [k]^n$, then $w_S(\sigma) = 1/k^n$ and $Z(S) = 1$.  Thus our goal will to enlarge $S$ until $S=[k]^n$, showing that $Z(S)$ can only decrease at each step.  
For a given $\sigma$ and $v$, let $\Cyl_v(\sigma)$ denote the set of $\tau \in [k]^n$ that we can obtain by letting $\sigma(v)$ vary arbitrarily:
\[
\Cyl_v(\sigma)= \left\{\tau \in [k]^n \mid \mbox{$\tau(u)=\sigma(u)$ for all $u \ne v$} \right\} \, .
\]
In particular, $\Cyl_v(\sigma) = \Cyl_v(\sigma')$ if and only if $\sigma' \in \partial_S(\sigma,v)$.  
Similarly, let $\Cyl_v(S)$ be the ``thickening'' of $S$ along the $v$ axis, 
\begin{align*}
\Cyl_v(S) 
&= \bigcup_{\sigma \in S} \Cyl_v(\sigma) \, . 
\end{align*}
We claim that this thickening can only decrease $Z$.  That is, for any $S \ne \emptyset$ and any $v$, 
\[
Z(S) \ge Z(\Cyl_v S) \, . 
\]
To see this, let $T=\Cyl_v(S)$.  Each $\sigma \in S$ contributes $c_v(\sigma,v)$ times to the union $\bigcup_{\sigma \in S} \Cyl_v(\sigma)$, so 
\begin{equation}
\label{eq:zt}
Z(T) = \sum_{\sigma \in S} \frac{1}{c_S(\sigma,v)} \,w_T(\Cyl_v(\sigma)) \, . 
\end{equation}
Since each $\sigma \in T$ has $c_T(\sigma,v)=k$, each $\tau \in \Cyl_v(\sigma)$ has $c_T(\tau,u) \ge c_S(\sigma,u)$ for all $u \ne v$, and $|\!\Cyl_v(\sigma)| = k$, we have
\[
w_T(\Cyl_v(\sigma)) \le k \,\frac{c_S(\sigma,v)}{k} \,w_S(\sigma) = c_S(\sigma,v) \,w_S(\sigma) \, .
\]
Combining this with~\eqref{eq:zt} gives
\[
Z(T) = \sum_{\sigma \in S} \frac{1}{c_S(\sigma,v)} \,w_T(\Cyl_v(\sigma)) \le \sum_{\sigma \in S} w_S(\sigma) = Z(S) \, . 
\]
To complete the proof, let $T_0 = S$, and for each $1 \le v \le n$ let $T_v = \Cyl_v(T_{v-1})$.  Then $T_n = [k]^n$, and
\[
Z(S) = Z(T_0) \ge Z(T_1) \ge \cdots \ge Z(T_n) = 1 \, . \qedhere
\]
\end{proof}

\begin{proof}[Proof of Lemma~\ref{lem:iso}.]
Let $S$ be the set of permuted $k$-colorings.  The number of available colors $c(\sigma,v)$ we defined in Section~\ref{sec:first} is almost identical to $c_S(\sigma,v)$ as defined in Lemma~\ref{lem:iso0}.  The only difference is that in Section~\ref{sec:first}, if $v$ has a self-loop then it forbids two of its own colors, namely $\pi_{v,v}(\sigma(v))$ and $\pi_{v,v}^{-1}(v)$.  Removing these self-loops can only increase $c(\sigma,v)$ and thus decrease $Z$, so if $S \ne \emptyset$ then $Z \ge 1$ by Lemma~\ref{lem:iso0}.
\end{proof}

%
%

\section{A little calculus}
\label{sec:calculus}

\begin{proof}[Proof of Lemma~\ref{lem:psi}]
Recall that
\[
\phi(\zeta) = h(\zeta) + (1-\zeta) \ln (k-1) - \ln k + \frac{d}{2} \ln \frac{p(\zeta)}{(1-1/k)^2} \, .
\]
We upper bound $\phi(\zeta)$ with a simpler function.  We have
\[
\frac{p(\zeta)}{(1-1/k)^2} - 1 = \frac{(k \zeta-1)^2}{(k-1)^3} \le \frac{1}{k-1} \, , 
\]
and if $-1 < x \le 1/(k-1)$ the third-order Taylor series gives
\[
\ln (1+x) \le x - \frac{x^2}{2} + \frac{x^3}{3}  = x - \frac{x^2}{2} \left(1-\frac{2x}{3}\right) \le x - \frac{x^2}{2} \left( 1-\frac{2}{3(k-1)} \right) 
:= \ell(x) \, . 
\]
Therefore, we have $\phi(\zeta) \le \psi(\zeta)$ where 
\[
\psi(\zeta) = h(\zeta) + (1-\zeta) \ln (k-1) - \ln k + \frac{d}{2} \,\ell\!\left( \frac{(k \zeta-1)^2}{(k-1)^3} \right) \, . 
\]
This upper bound is tight at $\zeta = 1/k$, where
\[
\psi(1/k) = \phi(1/k) = 0 \, . 
\]
We remark that using the first-order Taylor series $\ln (1+x) \le x$, or equivalently $\ell(x) = x$, yields a weaker lower bound on $d_k$, about $\ln k$ below the first moment upper bound.  

Except for the entropy function, the dependence of $\psi(\zeta)$ on $\zeta$ is polynomial, making its derivatives significantly simpler than those of $\phi$.  First we note that
\[
\psi''(1/k) = \frac{k^2}{(k-1)^3} \,\left(d-(k-1)^2 \right) \, .
\]
Thus if $d < (k-1)^2$, we have
\[
\phi''(1/k) \le \psi''(1/k) < 0 \, .
\]
Next, if $k \ge 2$ then the fourth derivative of $\psi(\zeta)$ is negative throughout the unit interval, 
\[
\psi''''(\zeta) = - 2 \left( \frac{1}{\zeta^3} + \frac{1}{(1-\zeta)^3} + d k^4 \frac{3k-5}{(k-1)^7} \right) < 0 \, . 
\]
As a consequence, $\psi(\zeta)$ has at most two local maxima in the unit interval, one of which is at $\zeta = 1/k$.  Our goal is to locate the other local maximum, which we denote $\zeta_2$, and to show that $\psi(\zeta_2) < 0$.

Assume that $d$ is $2+\eps$ below the first moment upper bound for some constant $\eps$, 
\[
d = 2 k \ln k - \ln k - 2 - \eps \, ,
\]
and set
\[
\zeta_2 = 1-\frac{a}{k} 
\]
for a constant $a$.  Then using Taylor series gives
\begin{align*}
\psi'(1-a/k) 
&=  \ln \frac{a}{k} - \ln \left(1-\frac{a}{k} \right) - \ln (k-1) + \frac{dk(k-1-a)}{(k-1)^3} \,\ell'\!\left( \frac{(k-1-a)^2}{(k-1)^3} \right) \\
&= \ln a - 2 \ln k + d \left( \frac{1}{k} + \frac{1-a}{k^2} \right) + O(1/k) \\
&= \ln a + (1-2a) \frac{\ln k}{k} + O(1/k) \, . 
\end{align*}
For any constant $a \ne 0$, for sufficiently large $k$ this is positive if $a > 1$ and negative if $a < 1$.  Therefore, $\zeta_2 = 1-a/k$ where $a$ tends to $1$ as $k \to \infty$, roughly as $a \approx k^{1/k}$.

Finally, again setting $d= 2 k \ln k - \ln k - 2 - \eps$, some more Taylor series give
\begin{align*}
\psi(1-a/k) 
&= h(a/k) + \frac{a}{k} \,\ln (k-1) - \ln k + \frac{d}{2} \,\ell\!\left( \frac{(k-1-a)^2}{(k-1)^3} \right) \\
&= \frac{1}{k} \left( a - a \ln a + 2 a \ln k \right) + O(1/k^2) - \ln k + \frac{d}{2} \left( \frac{1}{k} + \frac{1-4a}{2k^2} + O(1/k^3) \right) \\
&= \frac{1}{k} \left(a - a \ln a - 1 - \frac{\eps}{2} \right) + O\!\left( \frac{\ln k}{k^2} \right) \\
&\le -\frac{\eps}{2k} + O\!\left( \frac{\ln k}{k^2} \right) \, ,
\end{align*}
where we used $a - a \ln a \le 1$ for all $a > 0$.  Thus for any constants $a > 0$ and $\eps > 0$, when $k$ is sufficiently large $\psi(1-a/k)$ is negative, showing that $\psi(\zeta) < 0$ in the vicinity of $\zeta_2$.  This completes the proof.
\end{proof}

\begin{proof}[Proof of Lemma~\ref{lem:lfm}]
Recall that
\[
\lfm(d) 
= \frac{d}{2} \ln\left(1-\frac{1}{k}\right) + \frac{d}{k-1} + \ln \left( 1-\left(1-\e^{-d/(k-1)}\right)^k\right) \, .
\]
Setting 
\[
d = 2k \ln k - \ln k - 1 + \eps \, , 
\]
we have
\[
\e^{-d/(k-1)} = k^{-2} \,k^{-1/(k-1)} \,\e^{(1-\eps)/(k-1)} 
= k^{-2} + O(k^{-3} \log k) \, . 
\]
Taking the binomial series to second order gives
\begin{align*}
\left(1-\e^{-d/(k-1)}\right)^k &= 1-k \e^{-d/(k-1)} + {k \choose 2} \,\e^{-2d/(k-1)} + O(k^3 \e^{-3d/(k-1)})\\
&=1-k \e^{-d/(k-1)} + \frac{k(k-1)}{2} \,\e^{-2d/(k-1)} + O(k^{-3}) \, . 
\end{align*}
Plugging this into the last term of $\lfm(d)$ and using the Taylor series for $\ln (1-x)$ gives
\begin{align*}
\ln\left( 1-\left(1-\e^{-d/(k-1)}\right)^k\right) 
&= \ln\left( k \e^{-d/(k-1)} - \frac{k(k-1)}{2} \,\e^{-2d/(k-1)} + O(k^{-3})\right) \\
&= \ln k-\frac{d}{k-1}+ \ln \left(1 - \frac{(k-1)}{2} \,\e^{-d/(k-1)} + O(k^{-2}) \right) \\
&= \ln k-\frac{d}{k-1} - \frac{(k-1)}{2} \,\e^{-d/(k-1)} + O(k^{-2}) \\
&= \ln k-\frac{d}{k-1} - \frac{1}{2k} + O(k^{-2} \log k) \, ,
\end{align*}
and plugging this back in to $f(d)$ gives
\begin{equation}
\label{eq:f2}
\lfm(d) 
= \frac{d}{2} \ln\left(1-\frac{1}{k}\right) + \ln k - \frac{1}{2k} + O(k^{-2} \log k) \, . 
\end{equation}
The second order Taylor series for $\ln (1-x)$ gives
\begin{align*}
\frac{d}{2} \ln\left(1-\frac{1}{k}\right) 
&= \left(k\ln k - \frac{\ln k}{2} - \frac{1-\eps}{2}\right)\left(-\frac{1}{k} -\frac{1}{2k^2} +O(k^{-3})\right) \\
&= -\ln k + \frac{1-\eps}{2k} + O(k^{-2} \log k) \, . 
\end{align*}
Finally, putting this in~\eqref{eq:f2} gives
\[
f(d) = -\frac{\eps}{2k} + O(k^{-2} \log k) \, . 
\]
For any constant $\eps > 0$, this is negative for sufficiently large $k$, completing the proof.
\end{proof}

\section*{Acknowledgments.}  We benefited from conversations with Tom Hayes; with Lenka Zdeborov\'a and Florent Krz\c{a}kala on the Potts spin glass; and with Alex Russell and Dimitris Achlioptas on isoperimetric inequalities.  This work was partly supported by the McDonnell Foundation and the National Science Foundation.  Part of this work was done in 2008 while the third author was a student at Carnegie Mellon and a Research Experience for Undergraduates intern at the Santa Fe Institute.


\begin{thebibliography}{99}

\bibitem{ach-aco-rt} D. Achlioptas, A. Coja-Oghlan, and F. Ricci-Tersenghi, 
``On the solution-space geometry of random constraint satisfaction problems.''
\emph{Random Struct. Algorithms} 38(3): 251--268 (2011).

\bibitem{ach-molloy} D. Achlioptas and M. Molloy, ``Almost All Graphs with $2.522 n$ Edges are not $3$-Colorable.'' \emph{Electronic Journal of Combinatorics} 6 (1999) R29.

\bibitem{ach-moore} D. Achlioptas and C. Moore, ``Two moments suffice to cross a sharp threshold.''  \emph{SIAM Journal on Computing} 36 (2006) 740--762.

\bibitem{ach-moore-hyp} D. Achlioptas and C. Moore, ``On the two-colorability of random hypergraphs.''  \emph{Proc.\ 6th Intl.\ Workshop on Randomization and Approximation Techniques in Computer Science} (RANDOM '02) 78--90.

\bibitem{ach-moore-reg} D. Achlioptas and C. Moore, ``The chromatic number of random regular graphs.'' \emph{Proc.\ 8th Intl.\ Workshop on Randomization and Computation} (RANDOM '04), 219--228.

\bibitem{ach-naor} D. Achlioptas and A. Naor, ``The Two Possible Values of the Chromatic Number of a Random Graph.'' \emph{Ann. Math.} 162 (3), (2005), 1333--1349.

\bibitem{ach-peres} D. Achlioptas and Y. Peres, ``The Threshold for Random $k$-SAT is $2k \log 2 - O(k)$.''  \emph{J. AMS} 17 (2004) 947--973.

\bibitem{bhatnagar} Nayantara Bhatnagar, Juan Carlos Vera, Eric Vigoda, and Dror Weitz, 
``Reconstruction for Colorings on Trees.'' \emph{SIAM J. Discrete Math.} 25(2):809--826 (2011).

\bibitem{amin-catch} Amin Coja-Oghlan and Konstantinos Panagiotou, ``Catching the $k$-NAESAT threshold.''  Preprint, \url{arXiv:1111.1274v1}.

\bibitem{amin-lenka} Amin Coja-Oghlan and Lenka Zdeborov\'a, ``The condensation transition in random hypergraph 2-coloring.''  Preprint, \url{arXiv:1107.2341v2}.

\bibitem{dubois-mandler} O. Dubois and J. Mandler, ``On the non-3-colorability of random graphs.''  Preprint, \url{arXiv:math/0209087v1}.

\bibitem{kaporis} Alexis C. Kaporis, Lefteris M. Kirousis, and Yannis C. Stamatiou, ``A note on the non-colorability threshold of a random graph.'' \emph{Electronic Journal of Combinatorics} 7(1) (2000) R29.
 
\bibitem{krzakala:etal:pnas} 
Florent Krz\c{a}kala, Andrea Montanari, Federico Ricci-Tersenghi, Guilhem Semerjian, and Lenka Zdeborov\'a, 
``Gibbs states and the set of solutions of random constraint satisfaction problems.''
\emph{Proc. Natl. Acad. Sci.} 104(25):10318--10323 (2007).

\bibitem{lenka-potts} Florent Krz\c{a}kala and Lenka Zdeborov\'a, 
``Potts Glass on Random Graphs.''  \emph{Euro. Phys. Lett.} 81:57005 (2008).

\bibitem{maneva-sinclair} Elitza N. Maneva and Alistair Sinclair, ``On the satisfiability threshold and clustering of solutions of random 3-SAT formulas.''
\emph{Theoretical Computer Science} 407(1--3):359--369 (2008).

\bibitem{mertens-mezard-zecchina} S. Mertens, M. M{\'e}zard, and R. Zecchina, ``Threshold values of Random $k$-SAT from the cavity method.'' \emph{Random Structures and Algorithms} 28 (2006) 340--373.

\bibitem{mezard-zecchina} M. M{\'e}zard, G. Parisi, and R. Zecchina, ``Analytic and Algorithmic Solution of Random Satisfiability Problems.'' \emph{Science} 297 (2002) 812.

\bibitem{mitz-upf} M. Mitzenmacher and E.Upfal, 2005. ``Probability and 
Computing: Randomized Algorithms and Probabilistic Analysis'' Cambridge University Press, New York, NY, USA.

\bibitem{recon-cluster} Andrea Montanari, Ricardo Restrepo, and Prasad Tetali,
``Reconstruction and Clustering in Random Constraint Satisfaction Problems.''
\emph{SIAM J. Discrete Math.} 25(2):771--808 (2011).

\bibitem{mulet} R. Mulet, A. Pagnani, M. Weigt, and R. Zecchina,  
``Coloring random graphs.''  \emph{Phys. Rev. Lett.} 89 (2002) 268701. 

\bibitem{sly} Allan Sly, ``Reconstruction of Random Colourings.'' 
\emph{Communications in Mathematical Physics} 288(3):943--961 (2009).

\bibitem{lenka-boettcher} Lenka Zdeborov\'a and Stefan Boettcher, 
``Conjecture on the maximum cut and bisection width in random regular graphs.'' 
\emph{J. Stat. Mech.} (2010) P02020.

\bibitem{lenka-col} Lenka Zdeborov\'a and Florent Krz\c{a}kala, 
``Phase transitions in the coloring of random graphs.''  
\emph{Phys. Rev. E} 76:031131 (2007)




\end{thebibliography}
\end{document}